\documentclass[a4paper,12pt]{article}
\usepackage{times, url}
\usepackage{graphicx,amsmath,amssymb,amsthm,amsfonts,bm,float,cite,soul,color}
\newtheorem{theorem}{Theorem}[section]

\newtheorem{conjecture}{Conjecture}[section]
\newtheorem{definition}{Definition}[section]
\newtheorem{proposition}{Proposition}[section]
\newtheorem{remark}{Remark}[section]

\usepackage[none]{hyphenat}

\usepackage[center]{caption}
\newcommand{\intH}[2]{\mathcal{H}_{#1}\left(#2\right)}

\newcommand{\mygcd}[2]{\mathrm{gcd}\left( #1,#2 \right)}
\newcommand{\Franel}[1]{\mathcal{F}\left( #1 \right)}
\newcommand{\IP}[1]{\textbf{E}\left( #1 \right)}
\newcommand{\FP}[1]{\left\{ #1 \right\}}
\newcommand{\card}[1]{\left| #1 \right|}
\newcommand{\rad}[1]{\textbf{rad}\left( #1 \right)}
\newcommand{\setTot}[1]{t\left( #1 \right)}
\newcommand{\pef}[2]{\left\lfloor \frac{#1}{#2} \right\rfloor}

\newcommand{\doubleparii}[1]{\left(\!\!\left(#1\right)\!\!\right)}
\newcommand{\doublepari}[1]{\left(\!\left(#1\right)\!\right)}

\usepackage[bottom,symbol]{footmisc}
\DeclareFontEncoding{TS1}{}{}
\DeclareFontSubstitution{TS1}{cmr}{m}{n}
\DeclareTextSymbol{\tcrp}{TS1}{'251}
\DeclareTextSymbolDefault{\tcrp}{TS1}

\usepackage[hidelinks]{hyperref} 
\usepackage{orcidlink}

\hypersetup{pdftitle={Your paper title here in sentence case}, pdfauthor={First Author, Second Author}}

\begin{document}
\setcounter{page}{1}

\thispagestyle{empty} 

\noindent \textbf{Notes on Number Theory and Discrete
Mathematics \newline
Print ISSN 1310--5132, Online ISSN 2367--8275 \newline
XXXX, Volume XX, Number X, XXX--XXX \newline
DOI: 10.7546/nntdm.XXXX}  
\vspace{11mm}

\begin{center}
{\LARGE \bf On the calculation of an integral \\[4mm] introduced by Hubert Delange} 
\vspace{6mm}

{\Large \bf Steven Brown$^1$}
\vspace{4mm}

$^1$ 48 rue Pottier, 78150 Le Chesnay Rocquencourt, France \\
e-mail: \url{steven.brown.math@gmail.com}
\end{center}

\vspace{4mm}  

\noindent
\noindent {\bf Received:}  DD Month XXXX \hfill {\bf Revised:} DD Month XXXX \\
{\bf Accepted:} DD Month XXXX  \hfill {\bf Online First:} DD Month XXXX \\[4mm] 
{\bf Abstract:} 
In his paper, ``Sur la distribution des fractions irréductibles de dénominateur $n$ ou de dénominateur au plus égal à $x$'', Hubert Delange considers the family integrals defined by

\begin{equation*}
\int_0^1\left(\frac{N(n,t)}{\phi(n)}-t\right)^{2q}dt,
\end{equation*}
where $n>1$ and $q>0$ are positive integers, $\phi$ is the Euler totient function, and for any real number $t>0$, $N(n,t)$ counts the number reduced fractions whith denominator $n$ that belong to the interval $]0,t]$. He was able to evaluate the integral for $q=1$ but stated that ``\emph{Malheureusement, nous n'avons pas su jusqu'à présent faire cette évaluation pour $q>1$}''. In addition to its intrinsic interest, this calculation would help understand the asymptotic distribution of the $\phi(n)$ totatives of $n$ relative to $\phi(n)$ equidistributed numbers in $[0,n]$.

In this paper, we first provide some general properties of the integral for any $q>1$ and reduce the problem to its calculation for odd square-free integers with at least two distinct prime factors. We then analyse into more detail the case $q=2$ for which we provide a formula that reduces the problem to the evaluation of the fourfold Möbius sum, of integer valued modular sums related to generalized Dedekind sums.\\

\vspace{5mm}

\section{Introduction}

Delange considered the following family of integrals\footnote{Although the family depends on two positive integer $n> 1$ and $q>0$ we tend to see the problem as finding a formula valid for all $n$, for any fixed value of $q$.},
\[ \int_0^1\left(\frac{N(n,t)}{\phi(n)}-t\right)^{2q}dt, \]
with a view to study the asymptotic distribution of the differences
\[ a_j - \left(j-\frac{1}{2}\right)\frac{n}{\phi(n)}, \]
for $j$ ranging from 1 to $n$, where $a_j$, is the j\textsuperscript{th} totative\footnote{Or totitive as introduced by James Joseph Sylvester in 1879, see Chapter V in \cite{dickson1919history}. See also Chapter 3 in \cite{sandor2004handbook} for more recent results.} in ascending order. We denote  by $\setTot{n}$ the set of the $\phi(n)$ totatives of $n$. To date, it seems that the only known formula is
\begin{equation}\label{eq:Delange_integral_q_egal_1}
\int_0^1\left(\frac{N(n,t)}{\phi(n)}-t\right)^{2}dt = \frac{2^{\omega(n)}}{12n\phi(n)},
\end{equation}
when $q=1$, with $\omega(n)$ being the number of distinct prime divisors of $n$. It was first proved by Hubert Delange in \cite{delange1968distribution}, using multiplicative functions. The formula was proved later by Lucien Van Hamme in \cite{van1971generalisation}, who calculated this integral using Parseval's Theorem. Let us introduce the functions
\begin{equation}\label{eq:theta_n}
\theta_n(x):=\sum_{d\mid n}\mu(d)\IP{ \frac{x}{d} },
\end{equation}
and
\begin{equation}\label{eq:nu_n}
\nu_n(x):=\sum_{d\mid n}\mu(d)\FP{ \frac{x}{d} }.
\end{equation}
Given that $\theta_n(x)$ counts the number of totatives less than $x$ (see for example \cite{apostol1976introduction}, chapter 2, exercise 9 page 47), we have $N(n,t)=\theta_n(nt)$. Knowing that\footnote{Because we have $\mu(d)\frac{x}{d}=\mu(d)\IP{\frac{x}{d}} + \mu(d)\FP{\frac{x}{d}}$ and $\frac{\phi(n)}{n}=\sum_{d\mid n} \frac{\mu(d)}{d}$ from \cite{apostol1976introduction} page 26.}, for any real number $x$,
\begin{equation}\label{eq:formula_thetanx_plus_nunx}
\theta_n(x)+\nu_n(x) = x\frac{\phi(n)}{n},
\end{equation}  
we see that
\[
\int_0^1\left(\frac{N(n,t)}{\phi(n)}-t\right)^{2q} dt = \frac{\intH{2q}{n}}{\phi(n)^{2q}},
\]
where $\intH{k}{n}$ is defined by
\begin{equation}
\intH{k}{n}:=\int_0^1\left(-\nu_n(nt)\right)^kdt.
\label{eq:Hkn}
\end{equation}

This paper aims at studying the integral $\intH{k}{n}$, that is, to try to find closed form formulas of $\intH{k}{n}$, valid for any positive integer $n>1$, when the integer parameter $q>0$ is fixed.

\section{Some first considerations on \texorpdfstring{$\intH{k}{n}$}{H(k,n)}}

The integral $\intH{k}{n}$ is well defined since the integrand is a polynom of degree $k$ with only a finite number of $\phi(n)$ singularities (when $t=ad/n$ for some integer $d\mid n$ and for an integer $0<a<n/d$). We define the Franel integral (see \cite{mcintosh1996franel}) of order $k$ as
\begin{equation}\label{eq:def:Franel_k}
\Franel{a_1,\ldots ,a_k}:=\int_0^1 \doublepari{a_1 t}\cdots \doublepari{a_k t}dt,
\end{equation}
with $\doublepari{}$ defined by
\begin{equation}\label{eq:def_double_par}
 \doublepari{x}=
\begin{cases}
    \FP{x}-\frac{1}{2}, & \text{if $x\notin \mathbb{Z}$}\\
    0, & \text{otherwise}.
\end{cases}
\end{equation}
Note that we have\footnote{see ''An equivalent lattice point sum'' page 193 in \cite{mcintosh1996franel}}:
\begin{equation}\label{eq:Franel_sum_lattice_points}
\Franel{a_1,\ldots ,a_k} = \frac{1}{(2\pi)^k}\sum_{\substack{b_1,\ldots ,b_k\in \mathbf{Z}\setminus \left\{0\right\} \\ a_1b_1 + \cdots a_kb_k = 0}} \frac{1}{b_1\cdots b_k}.
\end{equation}

\begin{proposition}\label{prop:Hkn_is_Mobius_sum_of_Franel_integrals} $\intH{k}{n}$ is a k-fold Möbius sum of Franel integrals of order $k$.
\begin{equation}\label{eq:expr1_Hkn}
\intH{k}{n}:=(-1)^k\sum_{d_1\mid n}\mu(d_1)\cdots\sum_{d_k\mid n}\mu(d_k) \Franel{\frac{n}{d_1},\ldots,\frac{n}{d_k}},
\end{equation}
\end{proposition}

\begin{proof}
With $B(x)={x}-1/2$ the first periodic Bernoulli polynomial, we have $\nu_n(x)=\sum_{d\mid n}\mu(d)B(x)$ since $\sum_{d\mid n}\mu(d)=0$ (see \cite{hardy1979introduction}, Theorem 263). Therefore,
\[
\intH{k}{n}:=(-1)^k\sum_{d_1\mid n}\mu(d_1)\cdots\sum_{d_k\mid n}\mu(d_k) \int_0^1\prod_{i=1}^kB\left(\frac{n}{d_i} t\right)dt ,
\]
but
\begin{equation}\label{eq:Franel_with_B_or_dp}
\int_0^1\prod_{i=1}^kB\left(\frac{n}{d_i} t\right)dt=\int_0^1\prod_{i=1}^k\doubleparii{\frac{n}{d_i} t}dt,
\end{equation}
since the integrands are equal appart from a finite number of points.
\end{proof}

\begin{proposition}\label{prop:Hkn0_kodd} If $k$ is odd, $\intH{k}{n}=0$ for any positive integer $n$.
\end{proposition}

\begin{proof} A direct consequence of Franel integrals being equal to zero with an odd number of parameters in Equation (\ref{eq:expr1_Hkn}). In Equation (\ref{eq:Franel_sum_lattice_points}) when $a_1b_1+\cdots + a_kb_k=0$, we also have $a_1(-b_1)+\cdots + a_k(-b_k)=0$, and the two points lead to the contribution $\frac{1}{b_1\cdots b_k}+(-1)^k\frac{1}{b_1\cdots b_k}$ which is zero for odd $k$. Therefore all points contributions can be grouped by pairs that contribute to 0 in the sum.
\end{proof}

\begin{proposition}\label{prop:Hkn_egal_Hkradn} For any positive integers $k$ and $n$
\begin{equation}
\intH{k}{n}=\intH{k}{\rad{n}}.
\label{eq:Hkn_egal_Hkradn}
\end{equation}
\end{proposition}

\begin{proof} The Mobius function is zero on integers with a square factor, hence
\[
\intH{k}{n}:=(-1)^k\sum_{d_1\mid \rad{n}}\mu(d_1)\cdots\sum_{d_k\mid \rad{n}}\mu(d_k) \Franel{\frac{\rad{n}}{d_1}\frac{n}{\rad{n}},\ldots,\frac{\rad{n}}{d_k}\frac{n}{\rad{n}}},
\]
but $\Franel{\frac{\rad{n}}{d_1}\frac{n}{\rad{n}},\ldots,\frac{\rad{n}}{d_k}\frac{n}{\rad{n}}}=\Franel{\frac{\rad{n}}{d_1},\ldots,\frac{\rad{n}}{d_k}}$ as a general property of Franel integrals which is easily seen from Equation (\ref{eq:Franel_sum_lattice_points}): when $a_1b_1+\cdots + a_kb_k=0$ we also have $ca_1b_1+\cdots + ca_kb_k=0$ for any integer $c$.
\end{proof}

\begin{proposition} For any odd integer $m$,
\begin{equation}
\intH{k}{2m}=\intH{k}{m}.
\label{eq:Hk2m_egal_Hkm}
\end{equation}
\end{proposition}

\begin{proof} From Equation (\ref{eq:Hkn}), $\intH{k}{2m}=\int_0^1\left(-\nu_{2m}(2mt)\right)^kdt$. Since $\mu(2d)=-\mu(d)$, we see that
\[ \nu_{2m}(2mt) = \sum_{d\mid m}\mu(d)\left(B\left(2\frac{m}{d}t\right)-B\left(\frac{m}{d}t\right)\right).\]
We use lemma 1, page 78, in \cite{delange1968distribution}, with $n=2$, to see that $B(2x)-B(x)=B\left(x+\frac{1}{2}\right)$. Also, because $\frac{m}{d}$ is odd, whe have $B\left(\frac{m}{d}t+\frac{1}{2}\right)=B\left(\frac{m}{d}\left(t+\frac{1}{2}\right)\right)$. Thus $\nu_{2m}(2mt)=\nu_{m}\left(m\left(t+\frac{1}{2}\right)\right)$. The change of variable $u=t+\frac{1}{2}$, leads to $\intH{k}{2m}=\int_{\frac{1}{2}}^{\frac{3}{2}}\left(-\nu_{m}(mu)\right)^kdu$, which ends the proof, since $u\mapsto \nu_{m}(mu)$ has period 1.
\end{proof}

\subsection{An expression of \texorpdfstring{$\intH{k}{n}$}{H(k,n)} as a function of \texorpdfstring{$\setTot{n}$}{t(n)}}

We define 

\begin{equation}
V_k(n):=\sum_{a \in \setTot{n}} \nu_n(a)^k.
\label{eq:Vkn}
\end{equation}

\subsubsection{Preliminary results}

\begin{proposition}\label{prop:nuna_plus_nunma} Let $n$ be an integer greater than or equal to 2 and $a\in \setTot{n}$,
\begin{equation}\label{eq:nuna_plus_nunma}
\nu_n(a)+\nu_n(-a)=-1.
\end{equation} 
\end{proposition}

\begin{proof} let $d\mid n$. If $d=1$ then $\FP{\frac{a}{d}}+\FP{\frac{-a}{d}}=0$. If $d>1$, $\frac{a}{d}$ is not an integer and therefore $\FP{\frac{a}{d}}+\FP{\frac{-a}{d}}=1$. It follows that
\begin{equation*}
\nu_n(a)+\nu_n(-a) = \sum_{\substack{d\mid n\\ d>1}}\mu(d) = \left(\sum_{d\mid n}\mu(d)\right) - \mu(1) = -1.
\end{equation*}
\end{proof}

\begin{proposition}\label{prop:nunknpa} Let $n\geq 2$ be an integer, $k$ an integer and $a\in \setTot{n}$, 
\begin{equation}\label{eq:nunknpa}
\nu_n(kn+a)=\nu_n(a).
\end{equation} 
\end{proposition}

\begin{proof} It comes from a direct calculation and the fact that $k\frac{n}{d}$ is an integer when $d\mid n$:
\begin{equation*}
\nu_n(kn+a) =\sum_{d\mid n}\mu(d)\FP{k\frac{n}{d}+\frac{a}{d}} = \nu_n(a).
\end{equation*}
\end{proof}

\begin{theorem} \label{thm:Hkn_fn_Vkp1n} Let $k$ and $n$ be positive integers,
\begin{equation}
\intH{k}{n} = - \frac{1+(-1)^k}{(k+1)\phi(n)} V_{k+1}(n).
\label{eq:Hkn_fn_Vkp1n}
\end{equation}
\end{theorem}

\begin{proof}
The function $t\mapsto \left(\theta_n(nt)-\phi(n)t\right)^k$ is periodic of period 1, therefore for any real number $x$,
\begin{equation*}
\intH{k}{n}=\int_{x}^{x+1}\left(\theta_n(nt)-\phi(n)t\right)^kdt.
\end{equation*}
In particular, for $x=\frac{-1}{n}$, with the convention that $u_0=-1$,
\begin{align*}
\intH{k}{n}
&=\sum_{i=0}^{\phi(n)-1}\int_{\frac{u_i}{n}}^{\frac{u_{i+1}}{n}}\left(i-\phi(n)t\right)^kdt \\
&=\frac{-1}{(k+1)\phi(n)}\left(\sum_{i=0}^{\phi(n)-1} \left(i-\frac{\phi(n)}{n}u_{i+1}\right)^{k+1} - \sum_{i=0}^{\phi(n)-1} \left(i-\frac{\phi(n)}{n}u_{i}\right)^{k+1} \right).
\end{align*}
We shall look separately at the two sums, say respectively, $A$ and $B$. Regarding the first sum
\begin{align*}
A
&= \sum_{i=1}^{\phi(n)} \left(i-\frac{\phi(n)}{n}u_i-1\right)^{k+1}         &\text{change of index}\\
&= \sum_{a\in \setTot{n}} \left(\theta_n(a)-\frac{\phi(n)}{n}a-1\right)^{k+1} &\text{summing over }\setTot{n}\\
&= \sum_{a\in \setTot{n}} \left(-\nu_n(a)-1\right)^{k+1}                      &\text{from Equation (\ref{eq:formula_thetanx_plus_nunx})}\\
&= \sum_{a\in \setTot{n}} \nu_n(-a)^{k+1}                                     &\text{from Equation (\ref{eq:nuna_plus_nunma})}\\
&= \sum_{a\in \setTot{n}} \nu_n(n-a)^{k+1}                                    &\text{from Equation (\ref{eq:nunknpa})}\\
&= \sum_{a\in \setTot{n}} \nu_n(a)^{k+1}.                                     &\text{reverse order sum}
\end{align*}
Regarding the second sum, let's first remark that we have
\begin{equation*}
i+\phi(n)-\frac{\phi(n)}{n}u_{i+\phi(n)} = i-\frac{\phi(n)}{n}u_i,
\end{equation*}
hence the sum is constant provided it spans over $\phi(n)$ consecutive indexes.
\begin{align*}
B
&= \sum_{i=1}^{\phi(n)} \left(i-u_{i}\frac{\phi(n)}{n}\right)^{k+1}       &\text{from the remark above}\\
&= \sum_{a\in \setTot{n}} \left(\theta_n(a)-a\frac{\phi(n)}{n}\right)^{k+1} &\text{summing over }\setTot{n}\\
&= \sum_{a\in \setTot{n}} \left(-\nu_n(a)\right)^{k+1}                      &\text{from Equation (\ref{eq:formula_thetanx_plus_nunx})}\\
&= (-1)^{k+1} \sum_{a\in \setTot{n}} \nu_n(a)^{k+1}.                             &
\end{align*}
\end{proof}

\begin{remark} The formula implies that $\intH{k}{n}$ vanishes when $k$ is odd, consistently with Proposition \ref{prop:Hkn0_kodd}.
\end{remark}

\subsection{Calcul de \texorpdfstring{$\intH{k}{p}$}{H(k,p)}, for \texorpdfstring{$p$}{p} prime}

\begin{theorem} Let $p$ be a prime number and $k$ and $\alpha$ two positive integers, we have the following identity:
\begin{equation}
\intH{k}{p^\alpha} = \frac{1+(-1)^k}{(k+1)(k+2)(p-1)}\sum_{j=0}^{k+1}\binom{k+2}{j}B_j p^{1-j}.
\label{eq:Hkp_alpha}
\end{equation}
\end{theorem}

\begin{proof} From Proposition \ref{prop:Hkn_egal_Hkradn}, there is $\intH{k}{p^\alpha}=\intH{k}{p}$. Given that $p$ is a prime number, we have $\nu_p(i)=-\frac{i}{p}$ from Equation (\ref{eq:nu_n}). In that case, Equation (\ref{eq:Hkn_fn_Vkp1n}) simplifies to
\begin{equation*}
\intH{k}{p} = - \frac{1+(-1)^k}{(k+1)(p-1)} \left(\frac{-1}{p} \right)^{k+1} \sum_{i=1}^{p-1} i^k.
\end{equation*}
We get Equation (\ref{eq:Hkp_alpha}) after using Faulhaber's formula for sums of powers of consecutive integers that can be found in Equation 6.78 page 283 in \cite{graham1989donald}.
\end{proof}

In particular, some direct calculations lead to
\begin{equation}\label{eq:H2p_p_prime}
\intH{2}{p} = \frac{p-1}{6p},
\end{equation}
and
\begin{equation}\label{eq:H4p_p_prime}
\intH{4}{p} = \frac{p-1}{30}\left(\frac{2}{p}-\frac{2}{p^2}-\frac{1}{p^3}\right).
\end{equation}

To conclude this part, the problem has been reduced to the calculation of $\intH{k}{n}$ for $k$ an even integer greater than two and $n$ an odd, square-free integer, with at least two factors.

\subsection{Franel integral and Dedekind sums}\label{sec:Franel_and_Dedekind}

We finish this section with a general formula for the Franel integral with $k$ parameters. Let $E$ be a multiset of $k$ integers, we consider the Franel integral in its expression using Bernoulli functions:
\[
\Franel{E}=\int_0^1 \prod_{a\in E} B(a t)dt.
\]
We define\footnote{For simplicity in the notations, we may sometimes write $W_N(a_1,\ldots a_k)$ instead of $W_N(\{ a_1,\ldots a_k)\}$, since the formula does not depend on the variables order.}
\begin{equation}\label{eq:def:W_N}
W_N(E) := \sum_{r=1}^{N-1}\prod_{a\in E}(ar\bmod N).
\end{equation}
This definition, is an integer valued sum consistent with \cite{brown2026family}. The sum is closely related to generalised Dedekind\footnote{The exact dependency depends on the definition that we take for the generalised Dedekind sum, see for example \cite{zagier1973higher}. We shall keep this question outside of the scope of this paper.} sums since $N^{\card{E}}\prod_{a\in E}(ar\bmod N)=\prod_{a\in E}\FP{\frac{ar}{N}}$.

If $a\mid N$, then
\[
r\bmod\frac Na=\frac{ar\bmod N}{a},
\]
and we can also write that

\begin{equation}\label{eq:W_N_ecriture2}
W_N(E) = \left(\prod_{a\in E}a\right)\sum_{r=1}^{N-1}\prod_{a\in E}\left(r\bmod \frac{N}{a}\right).
\end{equation}

While McIntosh used a Fourier approach (see \cite{mcintosh1996franel}), we try here a naive approach that consists in looking for a subdivision of [0,1] such that on each interval of this subdivision, and for any $a\in E$, each function $t\mapsto B(at)$ is a linear function of $t$. Each of these functions have singularities when $t=b/a$ for an integer $b$. Therefore, any integer $N$ that is a common multiple of any integer of $E$, is such that the functions $B$ are linear on all intervals $]\frac{r}{N},\frac{r+1}{N}[$. For simplicity, we can take $N:=\text{lcm}(E)$, although any multiple of $N$ would be acceptable. The Franel integral may now be written as
\[
\Franel{E} = \sum_{r=0}^{N-1}\int_{\frac{r}{N}}^{\frac{r+1}{N}} \prod_{a\in E} B(at)dt.
\]
In each integral within the sum, we apply the change of variable $t=\frac{r}{N}+u$ and we get
\begin{equation*}
\Franel{E} = \sum_{r=0}^{N-1} \int_{u=0}^{\frac{1}{N}} \prod_{a\in E} B\left(\frac{ar}{N} + au\right) du.
\end{equation*}
Considering the Euclidean division of $r$ by the integer $\frac{N}{a}$ we can write that
\[
\frac{ar}{N}+au = \pef{ar}{N} + \frac{a}{N}\left(r\bmod \frac{N}{a}\right) + au.
\]
Given that $\pef{ar}{N}$ is an integer and given that
\[
0\leq \frac{a}{N}\left(r\bmod \frac{N}{a}\right) + au \leq \frac{a}{N}\left(\frac{N}{a}-1\right) + a \frac{1}{N} = 1,
\]
we can just remove the fractional part and write that
\begin{equation*}
B\left( \frac{ar}{N}+au \right) = \frac{a}{N}\left(r\bmod \frac{N}{a}\right) + au - \frac{1}{2}.
\end{equation*}
Then,
\begin{equation}\label{eq:Franel_fn_W}
\Franel{E} = \sum_{\substack{A\bigcup B =E\\A\bigcap B = \emptyset}} \frac{1}{N^{\card{A}}}W_N(A) \int_{u=0}^{\frac{1}{N}}\prod_{b \in B}\left( bu - \frac{1}{2}\right) du. 
\end{equation}
The integrals have polynomial integrands and can be calculated as
\begin{equation}\label{eq:formule_generale_integrale_polynomiale}
\int_{u=0}^{\frac{1}{N}}\prod_{b \in B}\left( bu - \frac{1}{2}\right) du = \sum_{C\subset B} {\left(\frac{-1}{2}\right)}^{\card{B}-\card{C}} \frac{\prod_{c\in C}c}{(\left|C \right|+1)N^{\left|C \right|+1}}.
\end{equation}

The arithmetical nature of the Franel integral lies in the evaluation of the sums $W_N(A)$, a problem that we think is intrinsically the evaluation of generalized Dedekind sums.

\section{Analysis of \texorpdfstring{$\intH{4}{n}$}{H(4,n)}}

\subsection{Exploitation of Theorem \ref{thm:Hkn_fn_Vkp1n}}

Theorem \ref{thm:Hkn_fn_Vkp1n} means that the calculation of $\intH{4}{n}$ is equivalent to the calculation $V_5(n)$. Indeed as we shall see, even $V_4(n)$ would be enough to calculate $\intH{4}{n}$. Unfortunately we have not found a formula for $V_4(n)$ but formulas for $V_i(n)$ are known for $i$ ranging from 0 to 3 included.

\begin{proposition} For any integer $n\geq 2$ and for any positive integer $k$
\begin{equation}
\sum_{a\in \setTot{n}} \left( {\nu_n(a)}^k + (-1)^{k+1}\left(1+\nu_n(a)\right)^k \right) = 0.
\label{eq:relation_liante_Vk}
\end{equation}
\end{proposition}

\begin{proof} Given that $a\mapsto n-a$ is a bijection of $\setTot{n}$, we have
\[
\sum_{a\in \setTot{n}} {\nu_n(a)}^k=\sum_{a\in \setTot{n}} {\nu_n(n-a)}^k,
\]
and then we use Propositions \ref{prop:nunknpa} and \ref{prop:nuna_plus_nunma}.
\end{proof}
The case $k=1$ provides
\begin{equation}\label{eq:rel_V1n_phin}
2V_1(n) + \phi(n) = 0.
\end{equation}
The case $k=3$ provides
\begin{equation}\label{eq:rel_V2n_V3n}
4V_3(n) + 6V_2(n) - \phi(n) = 0.
\end{equation}
The case $k=5$ gives
\begin{equation}\label{eq:rel_V5n_V4n_etc}
2V_5(n) + 5 V_4(n) + 10 V_3(n) + 10 V_2(n) + 5V_1(n) + \phi(n) = 0.
\end{equation}

\begin{remark} An even $k=2a$ for some integer $a$ happens to provide the same equation as $k=2a-1$ which is easily seen from Equation (\ref{eq:relation_liante_Vk}).
\end{remark}
It happens that we can calculate $V_k(n)$ for $k\leq3$, first of all 
\begin{equation}\label{eq:formula_V0n}
V_0(n)=\sum_{a\in \setTot{n}}{\nu_n(a)}^0=\sum_{a\in \setTot{n}}1=\phi(n).
\end{equation}
From Equation (\ref{eq:rel_V1n_phin}) we have
\begin{equation}\label{eq:formula_V1n}
V_1(n)=-\frac{\phi(n)}{2}.
\end{equation}
From Equation (\ref{eq:Hkn_fn_Vkp1n}), with $k=2$, we have
\begin{equation}
\intH{2}{n} = - \frac{2}{3\phi(n)} V_3(n).
\label{eq:expr1_J2}
\end{equation}
From Theorem 2 in \cite{delange1968distribution} we have
\begin{equation}
\intH{2}{n} = \frac{2^{\omega(n)}\phi(n)}{12n}.
\label{eq:expr2_J2}
\end{equation}
From Equations (\ref{eq:expr1_J2}) and (\ref{eq:expr2_J2}), it comes
\begin{equation}\label{eq:formula_V3n}
V_3(n) = - \frac{2^{\omega(n)}\phi(n)^2}{8n}.
\end{equation}
Using Equation (\ref{eq:rel_V2n_V3n}) we find that
\begin{equation}\label{eq:formula_V2n}
V_2(n) = \frac{2^{\omega(n)}\phi(n)^2}{12n} + \frac{\phi(n)}{6}.
\end{equation}

\begin{remark} From Theorem \ref{thm:Hkn_fn_Vkp1n}, we see that $\intH{4}{n}$ depends on $V_5(n)$. From Equation (\ref{eq:rel_V5n_V4n_etc}), we see that $V_5(n)$ depends on $V_i(n)$ for $0\leq i\leq 4$. The Equations (\ref{eq:formula_V0n}), (\ref{eq:formula_V1n}), (\ref{eq:formula_V2n}) and (\ref{eq:formula_V3n}) lead to
\begin{equation}\label{eq:relation_V5n_V4n}
V_5(n) = -\frac{5}{2}V_4(n)+\frac{\phi(n)}{12}\left( 5\cdot2^{\omega(n)-1}\frac{\phi(n)}{n}-1\right).
\end{equation}
This means that having a formula for either $V_4(n)$ or $V_5(n)$ would allow for the calculation $\intH{4}{n}$. 
\end{remark}

\subsection{A short proof of Spence formula}

Although the result below is not necessary in the scope of the paper, we believed that it was worth mentionning that the Equation (\ref{eq:formula_V2n}), which we obtained rather easily, allows for a short proof of Spence formula (see \cite{spence1963formulae},\cite{van1971generalisation},\cite{brown2026SpenceFormula}) which states that
\begin{equation*}
\sum_{j=1}^{\phi(n)}ja_j = \frac{\phi(n)}{24}\left(8n\phi(n)+6n+2\phi(m)(-1)^{\omega(m)}-2^{\omega(m)}\right),
\end{equation*}
where $m=\rad{n}$.

\begin{proof} The sum for $a\in \setTot{n}$ of the square of Equation (\ref{eq:formula_thetanx_plus_nunx}) evaluated at $x=a$ provides
\begin{equation*}
\sum_{a\in \setTot{n}}\theta_n(a)^2+\sum_{a\in \setTot{n}}\nu_n(a)^2+2\sum_{a\in \setTot{n}}\theta_n(a)\nu_n(a) = \frac{\phi(n)^2}{n^2} \sum_{a\in \setTot{n}} a^2.
\end{equation*}
All the sums but one are known\footnote{The first sum is just $\sum_{j=1}^{\phi(n)}j^2$; the second is Equation (\ref{eq:formula_V2n}); the last one, $\sum_{a\in \setTot{n}} a^2$ can be found in \cite{apostol1976introduction}.} and therefore we get a formula for the unknown sum
\begin{equation}\label{eq:formula_sum_thetananuna}
\sum_{a \in \setTot{n}}\theta_n(a)\nu_n(a)=\frac{{(-1)}^{\omega(n)}}{12n}\phi(n)^3-\left(\frac{1}{4}+\frac{2^{\omega(n)}}{24n}\right)\phi(n)^2-\frac{1}{6}\phi(n).
\end{equation}
Now, observe that
\begin{equation*}
\theta_n(a)a = \theta_n(a) \frac{n}{\phi(n)}\left(\theta_n(a)+\nu_n(a)\right),
\end{equation*}
and therefore, summing over $a\in \setTot{n}$
\[
\sum_{a \in \setTot{n}}\theta_n(a)a=\frac{n}{\phi(n)}\left(\sum_{a \in \setTot{n}}{\theta_n(a)}^2+\sum_{a \in \setTot{n}}\theta_n(a)\nu_n(a)\right),
\]
which allows for a calculation of Spence formula since the two sums on the right-hand side are known.
\end{proof}

\subsection{Evaluation of the Franel integrals of order 2 and 4}

Let $E$ be a multiset of 4 integers, with $N$ being their least common multiple. The evaluation of $\Franel{E}$, using Equation (\ref{eq:Franel_fn_W}) involves the evaluation of the sums $W_N(A)$ for $A\subset E$. The expression of $W_N$ as in Equation (\ref{eq:Franel_fn_W}) means that we only need to find formulas for sums of the form
\begin{equation}\label{eq:def:X_N_A}
X_N(A):=\sum_{r=0}^{N-1} \prod_{a \in A} (r\bmod a),
\end{equation}
where $A$ is a multiset of divisors of $N$ since the two sums are linked via the following equation
\begin{equation}\label{eq:W_N_function_of_X_N}
W_N(A)=\left(\prod_{a\in A}a\right)X_N\left(\left\{\frac{N}{a}\;\mid\;a\in A\right\}\right)
\end{equation}

\subsubsection{Evaluation of the sums \texorpdfstring{$X_N(A)$}{XN(A)}, \texorpdfstring{$\card{A}\leq 3$}{card(A)<=3}}

As we shall see below, we found closed form formulas for $X_N(A)$ for $\card{A}\leq 3$. We have not found a simple closed form formula for $\card{A}>3$. It is clear that 
\begin{equation}\label{eq:X_N_void}
X_N(\emptyset)=N.
\end{equation} 
and that
\begin{equation}\label{eq:X_N_a}
X_N(\{a\})=\frac{N}{2}(a-1).
\end{equation}
which are calculated as two elementary sums. The cases $\card{A}=2$ and $\card{A}=3$ are treated below:

\begin{proposition}\label{prop:formula_X_N_a_b} Let $a$, $b$ and $N$ be positive integers such that $a\mid N$ and $b\mid N$, we have the following formula
\begin{equation}\label{eq:X_N_a_b}
X_N(\{a,b\})=\frac{N}{4}(a-1)(b-1)+\frac{N}{12}\left(\mygcd{a}{b}^2-1\right).
\end{equation}
\end{proposition}

\begin{proof} Let $g=\mygcd{a}{b}$. We can write that $a=gu$ and $b=gv$ for some integers $u$ and $v$. Since $a\mid N$ and $g\mid a$, $g\mid N$ and there exists an integer $M$ such that $N=gM$. Now,
\begin{align*}
X_N(\{a,b\}) &= \sum_{q=0}^{M-1}\sum_{s=0}^{g-1} ((gq+s)\bmod gu)((gq+s)\bmod gv) \\
&= \sum_{q=0}^{M-1}\sum_{s=0}^{g-1} (g(q\bmod u)+s)(g(q\bmod v)+s).
\end{align*}
Note that, $a\mid N$ means that there exists an integer $\theta$ such that $N=a\theta=gu\theta$, but $N=gM$ thus $u\mid M$. Similarly $v\mid M$, but $\mygcd{u}{v}=1$ means that $uv\mid M$. Since $M$ is a multiple of $uv$ the sums over $q$ can easily be evaluated as repeated elementary sums and the result follows.
\end{proof}

\begin{remark} We see that the formula vanishes for $a=1$ or $b=1$, as expected. We also see that $a=b=N$ leads to the formula of a sum of squares, as expected.
\end{remark}

\begin{proposition}\label{prop:formula_X_N_a_b_c} Let $a$, $b$, $c$ and $N$ be positive integers such that $a\mid N$, $b\mid N$ and $c\mid N$, we have the following formula
\begin{multline}\label{eq:W_N_a_b_c}
X_N(\{a,b,c\})=\frac{N}{8}(a-1)(b-1)(c-1)+\frac{N}{24}(a-1)\left(\mygcd{b}{c}^2-1\right)+\\
\frac{N}{24}(b-1)\left(\mygcd{a}{c}^2-1\right)+\frac{N}{24}(c-1)\left(\mygcd{a}{b}^2-1\right).
\end{multline}

\end{proposition}

\begin{proof} We write the sum in the reverse order, observing that $(n-r)\bmod a=a-1-r\bmod a$. The result follows easily after development of the product under the sum. The same sum appears with a negative sign on the right-hand side. We then use the previous formulas for one and two modulos.
\end{proof}

\begin{remark} We see that the formula vanishes for $a=1$ or $b=1$, as expected. We also see that $a=b=N$ leads to the formula of a sum of cubes, as expected.
\end{remark}

We formulate the following conjecture\footnote{All our attemps to find plausible candidate functions via the resolution of systems of equations have failed. We believe this is consistent with the fact that no simple formula for the Franel integrals of order 4 is known in general.}:

\begin{conjecture} If $N$ is an integer and if $a$, $b$, $c$ and $d$ are divisors of $N$, then $X_N(a,b,c,d)$ does not have a polynomial closed form formula in $a$, $b$, $c$, $d$ and in any greatest common divisors of any combination of these variables, pairs, triplets or quadruplets.
\end{conjecture}

\subsubsection{An alternative proof of the formula of Franel integrals of order 2}

As concluded in Section \ref{sec:Franel_and_Dedekind}, a direct a calculation of the integral is possible if we have formulas for the expressions $W_N(A)$ since we already have Equation (\ref{eq:formule_generale_integrale_polynomiale}). For the evaluation of $\Franel{a,b}$, we only need to have formulas for $W_N(A)$ when $\card{A}\leq2$. Thanks to Equation (\ref{eq:W_N_function_of_X_N}), we only need to have formulas for $X_N(A)$ when $\card{A}\leq 2$. Given that these formulas were provided in the previous section, we see that a direct calculation is possible. With $E=\left\{a,b\right\}$, the Equations (\ref{eq:Franel_fn_W}) and (\ref{eq:formule_generale_integrale_polynomiale}) give immediately
\begin{multline*}
\Franel{a,b} = \frac{1}{N^{\card{\emptyset}}}W_N(\emptyset) \left(\frac{ab}{3N^3}-\frac{a+b}{4N^2}+\frac{1}{4N}\right)+\frac{1}{N^{\card{\{a\}}}}W_N(\{a\}) \left(\frac{b}{2N^2}-\frac{1}{2N}\right)\\
+ \frac{1}{N^{\card{\{b\}}}}W_N(\{b\}) \left(\frac{a}{2N^2}-\frac{1}{2N}\right) + \frac{1}{N^{\card{\{a,b\}}}}W_N(\{a,b\})\frac{1}{N}.
\end{multline*}
After replacing the $W_N$ by means of Equations (\ref{eq:W_N_function_of_X_N}), (\ref{eq:X_N_void}), (\ref{eq:X_N_a}) and (\ref{eq:X_N_a_b}), we get to
\begin{multline*}
\Franel{a,b} = \frac{ab}{3N^2}-\frac{a+b}{4N}+\frac{1}{4}+\\
\frac{1}{N}a\frac{N}{2}\left(\frac{N}{a}-1\right) \left(\frac{b}{2N^2}-\frac{1}{2N}\right)
+ \frac{1}{N} b\frac{N}{2}\left(\frac{N}{b}-1\right) \left(\frac{a}{2N^2}-\frac{1}{2N}\right) + \\
 \frac{1}{N^2} ab\left(\frac{N}{4}\left(\frac{N}{a}-1\right)\left(\frac{N}{b}-1\right)+\frac{N}{12}\left(\mygcd{\frac{N}{a}}{\frac{N}{b}}^2-1\right)\right) \frac{1}{N}.
\end{multline*}
After simplification, this expression simplifies to $\Franel{a,b}=\frac{{\mygcd{a}{b}}^2}{12ab}$, as expected (see \cite{franel1924suites}).

\subsubsection{Franel integrals of order 4}

We apply the same method as in the previous section at the order 4, with the multiset $E=\left\{a,b,c,d\right\}$ and $N$ any multiple of $\text{lcm}(E)$. The main difference with the order 2 is that we do not have a closed form fomula for $W_N(E)$, since $\card{E}=4$. Therefore, we suggest to keep this term as it is, and to replace all other terms $W_N(A)$ by their closed form formulas, when $A\subset E$ and $\card{A}\leq 3$. We have that

\begin{proposition}\label{prop:Franel4_as_function_of_W4}
Let $E=\left\{a,b,c,d\right\}$ be a multiset of 4 positive integers, and $N$ be any common multiple of $a,b,c,d$. For any unordered pair $\{x,y\}\subset E$, we write $\{z,t\}$ the unordered pair $E\setminus \{x,y\}$.
\begin{multline}\label{eq:Franel_ordre_4}
\Franel{a,b,c,d}=\frac{W_N(E)}{N^5}+\frac{abcd}{30N^4} -\frac{1}{72N^2}\sum_{\{x,y\}\subset E} \frac{{\mygcd{x}{y}}^2}{xy}zt + \\
\frac{1}{48N}\sum_{\{x,y\}\subset E} \frac{{\mygcd{x}{y}}^2}{xy}(z+t-M) -\frac{1}{24N^2}\sum_{\{x,y\}\subset E}xy +\frac{1}{16N}\sum_{x\in E}x-\frac{1}{16}.
\end{multline}
\end{proposition}

\subsubsection{Using Proposition \ref{prop:Franel4_as_function_of_W4} to express \texorpdfstring{$\intH{4}{N}$}{H4(N)}}

From Proposition \ref{prop:Hkn_is_Mobius_sum_of_Franel_integrals}, $\intH{4}{n}$ is a fourfold Möbius sum of Franel integrals of order 4 and the Equation (\ref{eq:Franel_ordre_4}) allows for some partial calculation of $\intH{4}{n}$. In addition, we use Proposition \ref{prop:Hkn_egal_Hkradn}, in order that we only need to calculate $\intH{4}{m}$ where $m=\rad{n}$ is square-free. In the expression on the right-hand side of Equation (\ref{eq:Franel_ordre_4}), all terms but one are either zero, or have a simple closed form formula. We denote by $\mathcal{W}_4(m)$ the contribution for which we do not know a closed form formula.

\begin{definition}\label{def:Somme_Mobius_W_4_m}
Let $m\ge2$ be a square-free integer; we define
\begin{equation}\label{eq:Somme_Mobius_W_4_m}
\mathcal{W}_4(m) := \frac{1}{m^5} \sum_{d_1,d_2,d_3,d_4\mid m} \mu(d_1) \mu(d_2) \mu(d_3) \mu(d_4) W_m\left( \frac m{d_1},\frac m{d_2},\frac m{d_3},\frac m{d_4} \right).
\end{equation}
\end{definition}

\begin{theorem}\label{thm:expression_H_4_m}
Let $m\ge2$ be a square-free integer. Then
\begin{equation}\label{eq:expression_H_4_m}
\intH{4}{m} = \mathcal{W}_4(m) + \frac{1}{60m}\left(\frac{\phi(m)}{m}\right)^3\left(2\phi(m)-5m2^{\omega(m)}\right).
\end{equation}
\end{theorem}

\begin{proof} If we consider Equation (\ref{eq:expr1_Hkn}), with $k=4$, we see that the same value $m$ can be used\footnote{Consistently with the comment made in section \ref{sec:Franel_and_Dedekind}.}, as a multiple of the least common multiple of the integers $\frac{n}{d_i}$ for $1\leq i\leq 4$. Now, looking at the right-hand side of Equation (\ref{eq:Franel_ordre_4}), it is clear that any term that does not depend simultaneously on $d_1$, $d_2$, $d_3$ and $d_4$ will have a zero fourfold Möbius sum, since we can factor at least one sum of the type $\sum_{d\mid m}\mu(d)$ which is equal to zero (See Theorem 263 in \cite{hardy1979introduction}). We immediately have
\[
\sum_{d_1,d_2,d_3,d_4\mid m} \mu(d_1) \mu(d_2) \mu(d_3) \mu(d_4) \frac{1}{30d_1d_2d_3d_4} = \frac{1}{30}\left(\frac{\phi(m)}{m}\right)^4.
\]
Also, the summand of the sum
\[
\sum_{d_1,d_2,d_3,d_4\mid m} \mu(d_1) \mu(d_2) \mu(d_3) \mu(d_4) {\mygcd{\frac{m}{d_1}}{\frac{m}{d_2}}}^2\frac{d_1d_2}{d_3d_4},
\]
is a multiplicative function, and there is no difficulty in proving that is equal to
\[
2^{\omega(m)}\frac{\phi(m)^3}{m}.
\]
There are 6 terms of that form, since there are 6 choices of unordered pairs in a set of 4 elements; their total contribution is 
\[
-\frac{1}{72m^2}\sum_{\{x,y\}\subset E} \frac{{\mygcd{x}{y}}^2}{xy}zt = -\frac{2^{\omega(m)}}{12}\left(\frac{\phi(m)}{m}\right)^3.
\]
which proves Equation (\ref{eq:expression_H_4_m}).
\end{proof}

\begin{remark} Let $p$ be a prime number, the Equation (\ref{eq:expression_H_4_m}) with $m=p$ gives
\begin{equation*}
\sum_{r=0}^{p-1}r^4 - \frac{1}{6}\left(\frac{p-1}{p}\right)^3 + \frac{1}{30}\left(\frac{p-1}{p}\right)^4,
\end{equation*}
which simplifies to 
\[
\frac{p-1}{30}\left(\frac{2}{p}-\frac{2}{p^2}-\frac{1}{p^3}\right),
\]
which is consistent with Equation (\ref{eq:H4p_p_prime}).
\end{remark}

\section{Conclusion}

The problem has not been solved in general, however it has been reduced to the calculation of $\intH{k}{m}$ for $k$ an even integer, greater than two, and $m$ an odd, square-free integer, with at least two factors. When $k=4$, the problem is equivalent to the calculation of $\mathcal{W}_4(m)$ or $V_4(m)$. In order to investigate this problem further, we suggest the following:
\begin{itemize}
	\item to use Parseval's theorem on $t\mapsto \nu_n(nt)^2$ which is 1-periodic, in an approach similar to that of Lucien Van Hamme (see \cite{van1971generalisation});
	\item a direct calculation of $V_4(m)$, if at all possible, using Spence's idea at the order 4 (see \cite{spence1963formulae});
	\item an exhaustive analysis of the properties of the sums $W_N$ and their possible reciprocity laws inherited from those of generalized Dedekind sums.
\end{itemize}

\section*{Acknowledgements} 

I would like to thank my wife Natallia for her support.

\bibliographystyle{plain}
\bibliography{biblio}

\end{document}